\documentclass{birkjour}
\usepackage[english]{babel}
\usepackage{amsmath,amssymb,amsthm,amscd}
\usepackage{paralist,prettyref} 
\usepackage{hyperref}
\usepackage{color}

\newtheorem{thm}{Theorem}[section]
\newtheorem{cor}[thm]{Corollary}

\newtheorem{prop}[thm]{Proposition}
\theoremstyle{definition}
\newtheorem{defn}[thm]{Definition}
\theoremstyle{remark}
\newtheorem{rem}[thm]{Remark}
\newtheorem{ex}[thm]{Example}
\numberwithin{equation}{section}

\newrefformat{cor}{Corollary~\ref{#1}}
\newrefformat{defn}{Definition~\ref{#1}}
\newrefformat{eq}{(\ref{#1})}
\newrefformat{lem}{Lemma~\ref{#1}}
\newrefformat{prob}{problem~(\ref{#1})}
\newrefformat{prop}{Proposition~\ref{#1}}
\newrefformat{thm}{Theorem~\ref{#1}}
\newrefformat{rem}{Remark~\ref{#1}}

\newcommand{\bbE}{\mathbb E}

\newcommand{\bbN}{\mathbb N}

\newcommand{\bbR}{\mathbb R}

\newcommand{\calB}{\mathcal B}

\newcommand{\calH}{\mathcal H}

\newcommand{\calM}{\mathcal M}
\newcommand{\calMR}{\mathcal {MR}}

\newcommand{\calR}{\mathcal R}

\newcommand{\dist}{\operatorname{dist}} 
\renewcommand{\div}{\operatorname{div}} 

\begin{document}

\title[Weighted $L_p$-spaces]{Linear and quasilinear evolution equations in the context of weighted $L_p$-spaces}

\author[M. Wilke]{Mathias Wilke}
\address{
Martin-Luther-Universit\"{a}t Halle-Wittenberg\\
Naturwissenschaftliche Fakult\"{a}t II\\
Institut f\"{u}r Mathematik\\
06099 Halle (Saale)}
\email{mathias.wilke@mathematik.uni-halle.de}



\date{\today}

\begin{abstract}
In 2004, the article \emph{Maximal regularity for evolution equations in weighted $L_p$-spaces} by J. Pr\"{u}ss and G. Simonett \cite{PrSi04} has been published in \emph{Archiv der Mathematik}. We provide a survey of the main results of that article and outline some applications to semilinear and quasilinear parabolic evolution equations which illustrate their power.
\end{abstract}

\maketitle

\section{Introduction}

In this survey article, we review regularity results for abstract linear parabolic evolution equations of the form
\begin{equation}\label{eq:linEE}
\dot{u}(t)+Au(t)=f(t),\quad t>0,\quad u(0)=u_0,
\end{equation}
in the framework of time-weighted $L_p$-spaces
$$L_{p,\mu}(\bbR_+;X):=\{f\in L_{1,loc}(\bbR_+;X):[t\mapsto t^{1-\mu}f(t)]\in L_p(\bbR_+;X)\},$$
$1<p<\infty$, that emanated from the groundbreaking paper \cite{PrSi04}. Here $X$ is a Banach space, $A$ is a closed linear operator in $X$ with dense domain $D(A)$ and the data $(f,u_0)$ are given. In \cite{PrSi04}, the concept of $L_{p,\mu}$-maximal regularity for the operator $A$ has been introduced, see Definition \ref{def:MR} for details. One of the fundamental results of the article \cite{PrSi04} states that this concept is independent of the parameter $\mu$ as long as $\mu>1/p$. Note that in case $\mu=1$ one ends up in the classical (unweighted) $L_p$-maximal regularity class, see e.g. \cite{DorVen87,PrSo90,Weis01} which is just a selection.

We will furthermore consider quasilinear parabolic evolution equations
\begin{equation}\label{eq:quasilinEE}
\dot{u}(t)+A(u(t))u(t)=F(u(t)),\quad t>0,\quad u(0)=u_0,
\end{equation}
and apply the linear theory for \eqref{eq:linEE} from \cite{PrSi04}, combined with the contraction mapping principle, which yields the existence and uniqueness of a local-in-time solution of \eqref{eq:quasilinEE}. Moreover, the advantage of working in time-weighted $L_p$-spaces is underlined by some results on parabolic regularization and the global-in-time existence for solutions to \eqref{eq:quasilinEE}. In Sections \ref{sec:polynom} \& \ref{sec:critspaces} of this survey article, we consider nonlinearities with a certain growth and so-called critical spaces. 

In the last section we list some selected references which have been influenced by the article \cite{PrSi04}.

Finally, let us mention that for example in \cite{Ama95, Ang90, CleSi01}, weighted continuous function spaces of the form
$$[t\mapsto t^{1-\mu} f(t)]\in BUC((0,T];X),\quad \lim_{t\to 0_+}t^{1-\mu} \|f(t)\|_X=0$$
have been studied earlier. It is shown that this class allows for maximal regularity results if the Banach space $X$ is replaced by a continuous interpolation space $(X,D(A))_{\theta,\infty}^0$ of order $\theta\in (0,1)$.

\section{Linear evolution equations}

In this section we review and discuss the setting and some of the main results of the article \cite{PrSi04}.

\subsection{Weighted spaces}

For an arbitrary Banach space $X$ and for $1<p<\infty$, we define the weighted $L_p$-space $L_{p,\mu}(\bbR_+;X)$ by
$$L_{p,\mu}(\bbR_+;X):=\{f\in L_{1,loc}(\bbR_+;X):[t\mapsto t^{1-\mu}f(t)]\in L_p(\bbR_+;X)\},$$
where $\mu\in (1/p,1]$ and $\bbR_+:=(0,\infty)$. Note that $L_p(\bbR_+;X)$ is the classical Bochner-Lebesgue space and evidently, it holds that $L_{p,1}(\bbR_+;X)=L_p(\bbR_+;X)$.

The weighted Sobolev-space $W_{p,\mu}^1(\bbR_+;X)$ is accordingly defined
by
$$W_{p,\mu}^1(\bbR_+;X):=\{u\in L_{p,\mu}(\bbR_+;X)\cap W_{1,loc}^1(\bbR_+;X):\dot u\in L_{p,\mu}(\bbR_+;X)\}.$$
The spaces $L_{p,\mu}(\bbR_+;X)$ and $W_{p,\mu}^1(\bbR_+;X)$ are equipped with the norms
$$\|f\|_{L_{p,\mu}(\bbR_+;X)}:=\left(\int_0^\infty |t^{1-\mu}f(t)|^pdt\right)^{1/p}$$
and
$$\|u\|_{W_{p,\mu}^1(\bbR_+;X)}:=\left(\|u\|_{L_{p,\mu}(\bbR_+;X)}^p+\|\dot{u}\|_{L_{p,\mu}(\bbR_+;X)}^p\right)^{1/p},$$
respectively, which turn them into Banach spaces.

The restriction $\mu>1/p$ is motivated by several reasons which we collect in the following
\goodbreak
\begin{prop}\label{prop:1}
For all $1<p<\infty$ and all $\mu\in (1/p,1]$ it holds that
\begin{enumerate}
\item $L_{p,\mu}(\bbR_+;X)\hookrightarrow L_{1,loc}(\overline{\bbR_+};X)$;
\vspace{0.2cm}
\item $W_{p,\mu}^1(\bbR_+;X)\hookrightarrow W_{1,loc}^1(\overline{\bbR_+};X)$;
\vspace{0.2cm}
\item (Hardy's inequality) For all $f\in L_{p,\mu}(\bbR_+;X)$ the estimate
\begin{equation}\label{eq:Hardy}
\int_0^\infty\left|t^{-\mu}\int_0^t f(s) ds\right|^pdt\le\frac{1}{(\mu-1/p)^p}\int_0^\infty |t^{1-\mu}f(t)|^pdt
\end{equation}
is satisfied.
\end{enumerate}
\end{prop}
\begin{proof}
The proof of the first two assertions can be found in \cite[Lemma 2.1]{PrSi04} while for a proof of the last assertion, we refer to \cite[Lemma 3.4.5]{PruSim16}.
\end{proof}

\subsection{Maximal regularity}

Let $X_0$, $X_1$ be Banach spaces such that $X_1\hookrightarrow X_0$ and $X_1$ is dense in $X_0$. Suppose that $A:X_1\to X_0$ is a linear and closed operator. We consider the abstract evolution equation
\begin{equation}\label{eq:linEE_0}
\dot{u}(t)+Au(t)=f(t),\quad t>0,\quad u(0)=0.
\end{equation}
Following the lines of \cite{PrSi04}, we have the following
\begin{defn} \label{def:MR}
Let $1<p<\infty$ and $\mu\in (1/p,1]$.
The operator \textbf{$A$ has the property of maximal $L_{p,\mu}$-regularity in $X_0$} if for each $f\in L_{p,\mu}(\bbR_+;X_0)$ there exists a unique solution
$$u \in W_{p,\mu}^1(\bbR_+;X_0)\cap L_{p,\mu}(\bbR_+;X_1)$$
of \eqref{eq:linEE_0}. If this is the case, we write for short $A\in \calMR_{p,\mu}(X_0)$ and
$$\calMR_{p,1}(X_0)=:\calMR_p(X_0)$$
if $\mu=1$.
\end{defn}
The following important and fundamental theorem states that the concept of maximal $L_{p,\mu}$-regularity does not depend on $\mu\in (1/p,1]$.
\goodbreak
\begin{thm}[Pr\"{u}ss \& Simonett \cite{PrSi04}]\label{thm:charactMR}
For all $1<p<\infty$ and $\mu\in (1/p,1]$ the following assertions are equivalent:
\begin{enumerate}
\item $A\in \calMR_{p,\mu}(X_0)$;
\vspace{0.2cm}
\item $A\in \calMR_{p}(X_0)$.
\end{enumerate}
\end{thm}
The proof of this theorem, which can be found in \cite[Theorem 2.4]{PrSi04}, relies crucially on Hardy's inequality \eqref{eq:Hardy}, a sophisticated splitting of the solution to \eqref{eq:linEE_0} and some multiplier results as e.g. \cite[Proposition 2.3]{PrSi04}.
\begin{rem}
The class $\calMR_{p}(X_0)$ is independent of $p\in (1,\infty)$, see \cite{Sob64}. Therefore, the class $\calMR_{p,\mu}(X_0)$ enjoys this property as well.
\end{rem}
\begin{rem}
If $A\in\calMR_p(X_0)$, then $-A$ is the generator of an exponentially stable analytic semigroup in $X_0$, see e.g. \cite[Theorem 2.2]{Dore93}, \cite[Proposition 1.2]{JanBari} or \cite[Proposition 3.5.2]{PruSim16}. A characterization of $A\in \calMR_p(X_0)$ has been given by Weis in \cite{Weis01}. It is based on the concept of $\calR$-boundedness in case $X_0$ is additionally of class UMD.
\end{rem}
\goodbreak
We proceed with a selection of examples for operators belonging to the class $\calMR_p(X_0)$.
\begin{itemize}
\item If $X_0$ is a Hilbert space and $-A$ generates an exponentially stable analytic semigroup in $X_0$, then $A\in \calMR_p(X_0)$ (\cite[Theorem 3.5.7]{PruSim16}).
\vspace{0.2cm}
\item If $X_0$ is a real interpolation space, and $-A$ generates an exponentially stable analytic semigroup in $X_0$, then $A\in \calMR_p(X_0)$ (\cite[Theorem 3.5.8]{PruSim16}).
\vspace{0.2cm}
\item Suppose $\Omega\subset\bbR^n$ is open with compact boundary of class $C^{2m}$, $m\in\bbN$. (Uniformly) normally elliptic operators of order $2m$ in $\Omega$ with appropriate homogeneous boundary conditions (satisfying the Lopatinskii-Shapiro condition) on $\partial\Omega$, induce operators $A\in \calMR_p(X_0)$ for $X_0=L_q(\Omega;E)$, $1<q<\infty$, where $E$ is a Banach space of class UMD. (\cite[Chapter 5]{DHP1} or \cite[Chapter 6]{PruSim16}).
\end{itemize}

\subsection{Trace spaces}

We turn our attention to the abstract Cauchy problem \eqref{eq:linEE} and ask under which assumptions on the initial value $u_0$, there exists a unique solution
$$u \in W_{p,\mu}^1(\bbR_+;X_0)\cap L_{p,\mu}(\bbR_+;X_1)$$
of \eqref{eq:linEE}, provided $A\in \calMR_{p}(X_0)$ and $f\in L_{p,\mu}(\bbR_+;X_0)$. By the second assertion in Proposition \ref{prop:1}, the trace operator
$${\rm tr}:W_{p,\mu}^1(\bbR_+;X_0)\to X_0,\quad u\mapsto u(0)$$
is well-defined, since
$$W_{1,loc}^1(\overline{\bbR_+};X_0)\hookrightarrow C(\overline{\bbR_+};X_0).$$
In the sequel, let
$$\mathbb{E}_{1,\mu}:=W_{p,\mu}^1(\bbR_+;X_0)\cap L_{p,\mu}(\bbR_+;X_1).$$
The following characterization of the trace space ${\rm tr}\ \mathbb{E}_{1,\mu}$ holds.
\begin{prop}\label{prop:trace}
For all $1<p<\infty$ and all $\mu\in (1/p,1]$ it holds that
$${\rm tr}\ \mathbb{E}_{1,\mu}=(X_0,X_1)_{\mu-1/p,p}$$
(up to equivalent norms), where $(X_0,X_1)_{\mu-1/p,p}$ is the real interpolation space between $X_0$ and $X_1$ of exponent $\mu-1/p$.
Moreover, the embedding
$$W_{p,\mu}^1(\bbR_+;X_0)\cap L_{p,\mu}(\bbR_+;X_1)\hookrightarrow BUC(\overline{\bbR_+};(X_0,X_1)_{\mu-1/p,p})$$
is true. Here $BUC$ stands for the \emph{bounded and uniformly continuous} functions.
\end{prop}
\begin{proof}
\cite[Proposition 3.1]{PrSi04}
\end{proof}
With the help of this proposition, one can prove the following
\begin{thm}[Pr\"{u}ss \& Simonett \cite{PrSi04}]\label{thm:lin}
Let $1<p<\infty$, $\mu\in (1/p,1]$ and $A\in \calMR_p(X_0)$. Then for all
$$(f,u_0)\in L_{p,\mu}(\bbR_+;X_0)\times (X_0,X_1)_{\mu-1/p,p}$$
there exists a unique solution
$$u \in W_{p,\mu}^1(\bbR_+;X_0)\cap L_{p,\mu}(\bbR_+;X_1)$$
of \eqref{eq:linEE}. Moreover, there exists a constant $C>0$ being independent of $(f,u_0)$ such that the estimate
\begin{equation}\label{eq:MRest}
\|u\|_{\mathbb{E}_{1,\mu}}\le C\left(\|f\|_{L_{p,\mu}(\bbR_+;X_0)}+\|u_0\|_{(X_0,X_1)_{\mu-1/p,p}}\right)
\end{equation}
holds.
\end{thm}
\begin{proof}
For the proof of the first assertion, see \cite[Theorem 3.2]{PrSi04}. Concerning the estimate \eqref{eq:MRest}, note that
$$\|(\dot{u}+Au,u(0))\|_{L_{p,\mu}(\bbR_+;X_0)\times (X_0,X_1)_{\mu-1/p,p}}\le C\|u\|_{\mathbb{E}_{1,\mu}}$$
by Proposition \ref{prop:trace}. Hence \eqref{eq:MRest} follows from the first assertion and the open mapping theorem.
\end{proof}
\begin{rem}
Theorem \ref{thm:lin} asserts in particular, that the regularity of the initial value can be reduced by decreasing the exponent $\mu\in (1/p,1]$ of the time-weight. This in turn implies that the number of compatibility conditions in the context of initial-boundary value problems for parabolic partial differential equations may be reduced to a minimum.
\end{rem}

\section{Quasilinear parabolic evolution equations}

\subsection{Local well-posedness and regularization}

In this section, we consider quasilinear evolution equations of the form
\begin{equation}\label{eq:quasilinEE2}
\dot{u}(t)+A(u(t))u(t)=F(u(t)),\quad t>0,\quad u(0)=u_1.
\end{equation}
We are looking for solutions in the maximal regularity class
$$u \in W_{p,\mu}^1((0,T);X_0)\cap L_{p,\mu}((0,T);X_1)=:\mathbb{E}_{1,\mu}(0,T)$$
with an appropriate $T\in (0,\infty)$. Therefore, by Proposition \ref{prop:trace}, it is reasonable to assume
\begin{equation}\label{eq:Ass_A_F}
(A,F)\in C^{1-}(V_\mu;\mathcal{B}(X_1,X_0)\times X_0)
\end{equation}
where $V_\mu\subset  (X_0,X_1)_{\mu-1/p,p}=:X_{\gamma,\mu}$ is open and nonempty, $C^{1-}$ stands for the \emph{locally Lipschitz continuous} functions and $\mathcal{B}(X_1,X_0)$ denotes the space of all bounded and linear operators from $X_1$ to $X_0$.

To exploit the general strategy, assume for simplicity that $V_\mu=X_{\gamma,\mu}$. Given any $u_0,u_1\in V_\mu$, we consider for given $w \in \mathbb{E}_{1,\mu}(0,T)$
the \emph{linear} problem
$$\dot{v}(t)+A(u_0)v(t)=G(w(t)),\quad t>0,\quad u(0)=u_1,$$
where $G(w):=F(w)+A(u_0)w-A(w)w$.
By the regularity assumption \eqref{eq:Ass_A_F} it follows that
$$[t\mapsto G(w(t))]\in L_{p,\mu}((0,T);X_0).$$
Assuming $A(u_0)\in \calMR_p(X_0)$, by Theorem \ref{thm:lin} and extension-restriction arguments, the mapping
$$[w\mapsto v]:\mathbb{E}_{1,\mu}(0,T)\to \mathbb{E}_{1,\mu}(0,T)$$
is well defined. Obviously, any fixed point of this mapping is a solution to \eqref{eq:quasilinEE2} and vice versa. The contraction mapping principle then yields the following
\begin{thm}\label{thm:LWP}
Let $p\in (1,\infty)$, $u_0\in V_\mu$ be given and suppose that $(A,F)$ satisfy \eqref{eq:Ass_A_F} for some $\mu\in (1/p,1]$. Assume in addition that $A(u_0)\in\calM\calR_{p}(X_0)$. Then there exist $T=T(u_0)>0$ and $\varepsilon=\varepsilon(u_0)>0$, such that $\bar{B}_{X_{\gamma,\mu}}(u_0,\varepsilon)\subset V_\mu$ and such that problem \eqref{eq:quasilinEE2} has a unique solution
    $$u(\cdot,u_1)\in W_{p,\mu}^1((0,T);X_0)\cap L_{p,\mu}((0,T);X_1)\cap C([0,T];V_\mu),$$
on $[0,T]$, for any initial value $u_1\in \bar{B}_{X_{\gamma,\mu}}(u_0,\varepsilon)$. Furthermore there exists a constant $c=c(u_0)>0$ such that for all $u_1,u_2\in \bar{B}_{X_{\gamma,\mu}}(u_0,\varepsilon)$ the estimate
    $$\|u(\cdot,u_1)-u(\cdot,u_2)\|_{\mathbb{E}_{1,\mu}(0,T)}\le c\|u_1-u_2\|_{X_{\gamma,\mu}}$$
is valid. Here $\bar{B}_{X_{\gamma,\mu}}(u_0,\varepsilon)$ denotes the closed ball with with center $u_0$ and radius $\varepsilon$ in the topology of $X_{\gamma,\mu}$.
\end{thm}
\begin{proof}
\cite[Theorem 2.1]{KPW10}
\end{proof}
\begin{rem}
A benefit of Theorem \ref{thm:LWP} is that the local existence time $T=T(u_0)$ is locally uniform.
Moreover, Theorem \ref{thm:LWP} shows that the space $X_{\gamma,\mu}=(X_0,X_1)_{\mu-1/p,p}$ is the natural phase space for the semi-flow $[u_0\mapsto u(t,u_0)]$ generated by \eqref{eq:quasilinEE} or \eqref{eq:quasilinEE2}.
In the unweighted case (i.e. $\mu=1$), Theorem \ref{thm:LWP} has been proven in \cite{CleLi93} \& \cite{JanBari}.
\end{rem}
\noindent
Concerning the continuation of solutions in the weighted maximal regularity class, one can prove the following
\begin{cor}\label{cor:maxInt}
Let the assumptions of Theorem \ref{thm:LWP} be satisfied and assume that $A(v)\in \calM\calR_{p}(X_0)$ for all $v\in V_\mu$. Then the solution $u(\cdot,u_0)$ of \eqref{eq:quasilinEE} has a maximal interval of existence $[0,t^+(u_0))$ and for each $T\in (0,t^+(u_0))$ there holds
$$u(\cdot,u_0)\in \mathbb{E}_{1,\mu}(0,T)\cap C([0,T];X_{\gamma,\mu}).$$
The mapping $[u_0\mapsto t^+(u_0)]:V_\mu\to (0,\infty)$ is lower-semicontinuous.
\end{cor}
\begin{proof}
\cite[Corollary 2.2]{KPW10} and \cite[Corollary 2.2]{LPW14}.
\end{proof}
Let us point out another advantage of working in the setting of weighted $L_p$-spaces. To see the benefit, observe that for all $\tau,T\in (0,t^+(u_0))$ with $\tau<T$, the estimate
\begin{equation}\label{eq:regularize1}
\tau^{1-\mu}\|u\|_{\mathbb{E}_{1,1}(\tau,T)}\le \|u\|_{\mathbb{E}_{1,\mu}(\tau,T)}\le  \|u\|_{\mathbb{E}_{1,\mu}(0,T)}
\end{equation}
for the solution $u$ of \eqref{eq:quasilinEE2} holds, hence
\begin{multline}\label{eq:regularize2}
u\in W^1_{p,loc}((0,t^+(u_0));X_0)\cap L_{p,loc}((0,t^+(u_0)),X_1)\\
\hookrightarrow C((0,t^+(u_0));X_{\gamma,1}),
\end{multline}
by Proposition \ref{prop:trace}.
This shows that the solution $u(t)$ of \eqref{eq:quasilinEE} with initial value $u_0\in X_{\gamma,\mu}=(X_0,X_1)_{\mu-1/p,p}$ regularizes instantaneously for $t\in (0,t^+(u_0))$ provided $\mu<1$.

\subsection{Global well-posedness}

Assume that for some $1/p<\mu<\nu\le 1$ it holds that
\begin{equation}\label{eq:compEmb}
X_{\gamma,\nu}\overset{c}{\hookrightarrow} X_{\gamma,\mu}
\end{equation}
and for some $\tau\in (0,t^+(u_0))$
$$u\in BC([\tau,t^+(u_0));X_{\gamma,\nu}),$$
where $\overset{c}{\hookrightarrow}$ stands for a \emph{compact embedding}.
Then, the orbit $\{u(t)\}_{t\in [\tau,t^+(u_0))}$ is relatively compact in $X_{\gamma,\mu}$ which in turn implies that $t^+(u_0)=\infty$ and therefore, the solution $u$ exists globally (see \cite[Section 3]{LPW14} for details). Moreover, combining the continuous dependence of the solution on the initial data from Theorem \ref{thm:LWP} with \eqref{eq:regularize1} \& \eqref{eq:regularize2}, one can even prove relative compactness of $\{u(t)\}_{t\in [\tau,t^+(u_0))}$ in $X_{\gamma,1}$, the \emph{strongest} topology with respect to $\mu\in (1/p,1]$. All these considerations are part of the following result which has been proven in \cite[Theorem 3.1]{LPW14}.
\begin{thm}\label{thm:GWP}
Let $p\in (1,\infty)$, suppose that $A(v)\in\calM\calR_{p}(X_0)$ for all $v\in V_\mu$ and let \eqref{eq:Ass_A_F} hold for some $\mu\in (1/p,1)$. Assume furthermore that \eqref{eq:compEmb} holds for some $\nu\in (\mu,1]$ and that the solution $u$ of \eqref{eq:quasilinEE} satisfies
    $$u\in BC([\tau,t^+(u_0));V_\mu\cap X_{\gamma,\nu})$$
for some $\tau\in (0,t^+(u_0))$ and $\nu\in (\mu,1]$ as well as
    $$\dist(u(t),\partial V_\mu)\ge \eta>0$$
for all $t\in [0,t^+(u_0))$. Then the solution $u$ exists globally and for each $\delta>0$, the orbit $\{u(t)\}_{t\ge \delta}$ is relatively compact in $X_{\gamma,1}$. If $u_0\in V_\mu\cap X_{\gamma,1}$, then $\{u(t)\}_{t\ge 0}$ is relatively compact in $X_{\gamma,1}$.

In addition, the $\omega$-limit set $\omega(u_0)$ defined by
$$\omega(u_0):=\left\{v\in V_\mu\cap X_{\gamma,1}:\ \exists\ t_n\nearrow\infty\ \mbox{s.t.}\ u(t_n,u_0)\to v\ \mbox{in}\ X_{\gamma,1}\right\}$$
is nonempty, compact and connected.
\end{thm}
\begin{ex}\label{ex:GWP}
We consider the quasilinear initial-boundary value problem
    \begin{equation}\label{appl1}
    \begin{aligned}
    \partial_t u-\div(a(u)\nabla u)&=f(u,\nabla u) &&\mbox{in}&&\Omega,\\
    u&=0&&\mbox{on}&&\partial\Omega,\\
    u(0)&=u_0&&\mbox{in}&&\Omega,
    \end{aligned}
    \end{equation}
where $\Omega\subset\bbR^n$ is bounded domain with boundary $\partial\Omega\in C^2$, $f\in C^{1}(\bbR\times\bbR^n;\bbR)$, $a\in C^{2}(\bbR;\bbR)$ and $a(s)\ge a_0>0$ for all $s\in\bbR$. Let us first rewrite \eqref{appl1} in the form \eqref{eq:quasilinEE}. To this end, for $1<q<\infty$, we set $X_0=L_q(\Omega)$,
    $$X_1=\{u\in W_q^2(\Omega):u|_{\partial\Omega}=0\},$$
where $u|_{\partial\Omega}$ has to be understood in the sense of traces. In this situation, we have for $\mu\in (1/p,1]$
    $$(X_0,X_1)_{\mu-1/p,p}=\begin{cases}
                                        \{u\in B_{qp}^{2\mu-2/p}(\Omega):u|_{\partial\Omega}=0\},&\ \mbox{if}\ \mu >1/p+1/(2q),\\
                                        B_{qp}^{2\mu-2/p}(\Omega),&\ \mbox{if}\ \mu<1/p+1/(2q),
                                        \end{cases}$$
see e.g. \cite{Gri69}. Let us assume that $2/p+n/q<1$, wherefore the embedding $B_{qp}^{2-2/p}(\Omega)\hookrightarrow C^1(\bar{\Omega})$ is at our disposal. In this case there exists $\mu_0\in (1/p,1)$ such that
    $$B_{qp}^{2-2/p}(\Omega)\overset{c}{\hookrightarrow} B_{qp}^{2\mu-2/p}(\Omega)\overset{c}{\hookrightarrow} C^1(\bar{\Omega}),\quad\mbox{if}\ \mu\in (\mu_0,1),$$
    with compact embeddings.
Indeed, the number $\mu_0\in (1/p,1)$ is given by
    $$\mu_0=\frac{1}{2}+\frac{1}{p}+\frac{n}{2q},$$
provided $2/p+n/q<1$. For $\mu\in (\mu_0,1]$ and $X_{\gamma,\mu}:=(X_0,X_1)_{\mu-1/p,p}$, we define $A:X_{\gamma,\mu}\to \calB(X_0,X_1)$ and $F:X_{\gamma,\mu}\to X_0$ by means of
    $$[A(v)u](x):=-\div(a(v(x))\nabla u(x)),\ x\in\Omega,\ v\in X_{\gamma,\mu},\ u\in X_1,$$
and
    $$F(v)(x):=f(v(x),\nabla v(x)),\ x\in\Omega,\ v\in X_{\gamma,\mu}.$$
From the regularity assumptions on $a$ and $f$ it follows that
    $$(A,F)\in C^{1-}(X_{\gamma,\mu};\calB(X_1,X_0)\times X_0),\quad \mu\in (\mu_0,1].$$
Furthermore, by \cite[Theorem 6.3.2]{PruSim16}, it holds that $A(v)\in \calMR_p(X_0)$ for all $v\in X_{\gamma,\mu}$, $\mu\in (\mu_0,1]$, since the spectral bound of $-A(v)$ is negative for each $v\in X_{\gamma,\mu}$. An application of Theorem \ref{thm:LWP}, Corollary \ref{cor:maxInt} and Theorem \ref{thm:GWP} yields the following result.
\begin{thm}
Let $n\in\bbN$, $p,q\in (1,\infty)$ such that $2/p+n/q<1$, $\Omega\subset\bbR^n$ a bounded domain with boundary $\partial\Omega\in C^2$. Suppose that $\mu>1/2+1/p+n/(2q)$
and let $u_0\in B_{qp}^{2\mu-2/p}(\Omega)$ such that $u_0|_{\partial\Omega}=0$. If the solution $u$
 of \eqref{appl1} satisfies
    $$u\in BC\left([\tau,t^+(u_0));B_{qp}^{2\nu-2/p}(\Omega)\right),$$
for some $\tau\in (0,t^+(u_0))$ and $\nu\in (\mu,1]$, then the solution $u$ exists globally, i.e.\ $t^+(u_0)=\infty$ and for any $\delta>0$, the set $\{u(t)\}_{t\ge \delta}$ is relatively compact in $B_{qp}^{2-2/p}(\Omega)$. Moreover, the $\omega$-limit set
    $$\omega(u_0):=\left\{v\in B_{qp}^{2-2/p}(\Omega):\ \exists\ t_n\nearrow\infty\ \mbox{s.t.}\ u(t_n;u_0)\to v\ \mbox{in}\ B_{qp}^{2-2/p}(\Omega)\right\}$$
is nonempty, connected and compact.
\end{thm}
\end{ex}

\subsection{Nonlinearities with polynomial growth}\label{sec:polynom}

The condition \eqref{eq:Ass_A_F} is for quite general functions $(A,F)$ and does not account for nonlinearities having a certain growth. Consider for example the PDE
\begin{equation}\label{eq:crit1}
\partial_t u-\div(a(u)\nabla u)=f(u).
\end{equation}
We compute (assuming that $a$ is sufficiently regular)
$$\div(a(u)\nabla u)=a(u)\Delta u+a'(u)|\nabla u|^2.$$
Therefore, we may also consider \eqref{eq:crit1} in the form
\begin{equation}\label{eq:crit2}
\partial_t u-a(u)\Delta u=f(u)+a'(u)|\nabla u|^2.
\end{equation}
Let us compare \eqref{eq:crit2} with the PDE
\begin{equation}\label{eq:crit3}
\partial_t u-a(u)\Delta u=f(u).
\end{equation}
For \eqref{eq:crit2} and \eqref{eq:crit3}, we choose the setting $X_0=L_q(\Omega)$ and
$$X_1=\{u\in W_q^2(\Omega):u|_{\partial\Omega}=0\}.$$
Then, as in Example \ref{ex:GWP}, the trace space is computed to the result
$$X_{\gamma,\mu}=\begin{cases}
                                        \{u\in B_{qp}^{2\mu-2/p}(\Omega):u|_{\partial\Omega}=0\},&\ \mbox{if}\ \mu >1/p+1/(2q),\\
                                        B_{qp}^{2\mu-2/p}(\Omega),&\ \mbox{if}\ \mu<1/p+1/(2q),
                                        \end{cases}$$
To solve \eqref{eq:crit3} in this setting, we require $X_{\gamma,\mu}\hookrightarrow C(\overline{\Omega})$, i.e. $2\mu>2/p+n/q$. Note that this is possible if $2/p+n/q<2$, in order to ensure $\mu\in (1/p,1]$. We turn back to \eqref{eq:crit2}. The terms $a(u)\Delta u$ and $f(u)$ are as in \eqref{eq:crit3} while the remaining term $a'(u)|\nabla u|^2$ seems to induce additional conditions on the weight $\mu$ for solving \eqref{eq:crit2}. However, it is of fundamental importance to observe that the last term in \eqref{eq:crit2} has a certain structure: it is of quadratic growth and of lower order compared to $a(u)\Delta u$.

Such a class of nonlinearities with a certain growth behaviour has been considered in \cite{LPW14}, \cite{PrSiWi18} and \cite{PrWi17}. We write
$$F(u)=F_r(u)+F_s(u)$$
where $F_r$ and $F_s$ are the regular and singular part of $F$, respectively. In the sequel, we denote by $X_\beta=(X_0,X_1)_\beta$, $\beta\in (0,1)$,  the complex interpolation spaces. The precise assumptions on $(A,F_r,F_s)$ are as follows.

\medskip
\noindent
{\bf (H1)} $(A,F_r)\in C^{1-}(V_\mu; \calB(X_1,X_0)\times X_0)$.

\medskip

\noindent
{\bf (H2)} $F_s: V_\mu\cap X_\beta \to X_0$ satisfies the estimate
$$
\|F_s(u_1)-F_s(u_2)\|_{X_0} \leq C \sum_{j=1}^m (1+\|u_1\|_{X_\beta}^{\rho_j}+\|u_2\|_{X_\beta}^{\rho_j})\|u_1-u_2\|_{X_{\beta_j}},
$$
$ u_1, u_2\in V_\mu\cap X_\beta$,
for some numbers $m\in\bbN$, $\rho_j\geq 0$, $\beta\in (\mu-1/p,1)$, $\beta_j\in [\mu-1/p, \beta]$, where $C$ denotes a constant which may depend on $\|u_i\|_{X_{\gamma,\mu}}$.
The case $\beta_j=\mu-1/p$ is only admissible if
{\bf (H2)} holds with $X_{\beta_j}$ replaced by $X_{\gamma,\mu}$.

\medskip

\noindent
{\bf (H3)} For all $j=1,\ldots,m$ we have
$$ \rho_j( \beta-(\mu-1/p)) + (\beta_j -(\mu-1/p)) \leq 1 -(\mu-1/p).$$
Allowing for equality in {\bf (H3)} is not for free and we additionally need to impose the following structural {\bf Condition (S)} on the Banach spaces $X_0$ and $X_1$.

\medskip

\noindent
{\bf (S)} The space $X_0$ is of class UMD and the embedding
$$ {W}^{1}_p(\bbR;X_0)\cap L_{p}(\bbR;X_1)\hookrightarrow {H}^{1-\beta}_{p}(\bbR;X_\beta),$$
is valid for each $\beta\in [0,1]$.
\goodbreak
\begin{rem}\mbox{}
\begin{enumerate}
\item By the {\em Mixed Derivative Theorem}, condition {\bf (S)} is in particular satisfied if $X_0$ is of class UMD, and if there is an operator $B$ in $X_0$, with domain
${\sf D}(B)=X_1$, such that $B$ possesses a bounded $\calH^\infty$-calculus in $X_0$ (for short  $B\in\calH^\infty(X_0)$)  with $\calH^\infty$-angle $\phi_{B}^\infty<\pi/2$. We refer to Pr\"uss and Simonett \cite[Section 4.5]{PruSim16} for details.
\item Hypothesis {\bf (H3)} with strict inequality has been considered in \cite{LPW14}. In this case, condition {\bf (S)} can be neglected.
\end{enumerate}
\end{rem}
\noindent
Let us apply {\bf (H2)} to the nonlinearity $F_s(u):=a'(u)|\nabla u|^2$ in \eqref{eq:crit2}. In this case, we have $m=1$, $\beta_1=\beta$ and $\rho_1=1$ and
$$2\beta=1+\frac{n}{2q}$$
to ensure the embedding $H_q^{2\beta}(\Omega)\hookrightarrow H_{2q}^1(\Omega)$.
From {\bf (H3)} it then follows that
$$2\mu\ge \frac{2}{p}+\frac{n}{q},$$
hence there are \emph{no additional conditions} on $\mu$ in \eqref{eq:crit2}.

The extension of Theorem \ref{thm:LWP} to the setting of this section reads as follows.
\begin{thm}\label{thm:criticalLWP} Suppose that the structural condition {\bf (S)} holds, and assume that hypotheses {\bf (H1), (H2), (H3)} are valid. Fix any $u_0\in V_\mu$ such that $A_0:= A(u_0)\in\calMR_p(X_0)$. Then there  is $T=T(u_0)>0$ and $\varepsilon =\varepsilon(u_0)>0$ with $\bar{B}_{X_{\gamma,\mu}}(u_0,\varepsilon) \subset V_\mu$ such that the problem
\begin{equation}\label{eq:quasilinEE3}
\dot{u}(t)+A(u(t))u(t)=F_r(u(t))+F_s(u(t)),\quad t>0,\quad u(0)=u_1,
\end{equation}
admits a unique solution
$$ u(\cdot, u_1)\in W^1_{p,\mu}((0,T);X_0)\cap L_{p,\mu}((0,T); X_1) \cap C([0,T]; V_\mu),$$
for each initial value $u_1\in \bar{B}_{X_{\gamma,\mu}}(u_0,\varepsilon)$. There is a constant $c= c(u_0)>0$ such that
$$ \|u(\cdot,u_1)-u(\cdot,u_2)\|_{\bbE_{1,\mu}(0,T)} \leq c\|u_1-u_2\|_{X_{\gamma,\mu}},$$
for all $u_1,u_2\in \bar{B}_{X_{\gamma,\mu}}(u_0,\varepsilon)$.
\end{thm}
\begin{proof}
\cite[Theorem 1.2]{PrSiWi18}
\end{proof}

\subsection{Critical spaces}\label{sec:critspaces}

Looking into the literature, there is no universally accepted definition of \emph{critical spaces}. One possible definition may be
based on the idea of a ‘largest space of initial data such that a given PDE is
well-posed. On the other hand, critical spaces are often introduced as ‘scaling invariant spaces,’ provided the underlying PDE enjoys a scaling.

We consider again the setting from Section \ref{sec:polynom}. Note that the condition {\bf (H3)} implies that the minimal value of $\mu$ is given by
$$ \mu_{c} = \frac{1}{p} + \beta -\min_j(1-\beta_j)/\rho_j,$$
which we call the {\em critical weight} as long as $\mu_c\in (1/p,1]$. Theorem \ref{thm:criticalLWP} shows in particular that we have local well-posedness of \eqref{eq:quasilinEE3} for initial values in
the spaces $X_{\gamma,\mu}$, provided {\bf (H1)} holds for $\mu\in [\mu_c,1]$. Therefore, it makes sense to name the space $X_{\gamma,\mu_c}=(X_0,X_1)_{\mu_c-1/p,p}$ the {\em critical space} for \eqref{eq:quasilinEE3}.
The critical space $X_{\gamma,\mu_c}$ enjoys the following properties.
\begin{itemize}
\item Generically, the critical space $X_{\gamma,\mu_c}$ is the largest space such that the underlying evolution equation is well-posed for initial values in $X_{\gamma,\mu_c}$. A concrete counterexample is given in \cite[Section 2.2]{PrSiWi18}.
\item If the underlying evolution equation admits a scaling, then the critical space $X_{\gamma,\mu_c}$ is scaling invariant. This has been proven in \cite[Section 2.3]{PrSiWi18}. A typical example is given by the Navier-Stokes equations in $\bbR^n$:
\begin{equation*}
\begin{aligned}
\partial_t u + u\cdot \nabla u -\Delta u +\nabla\pi =0,\quad
{\rm div}\, u =0,
\end{aligned}
\end{equation*}
which is invariant under the scaling
$$(u_\lambda (t,x),\pi_\lambda(t,x)):=(\lambda u(\lambda^2 t,\lambda x),\lambda^2 \pi(\lambda^2 t,\lambda x)).$$
Here $u$ denotes the velocity field and $\pi$ the pressure.
\item The critical spaces $X_{\gamma,\mu_c}$ are invariant with respect to interpolation-extrapolation scales (see e.g. \cite{Ama95} for the theory of those scales). This has been proven in \cite[Section 2.4]{PrSiWi18}. Considering a PDE in a scale of function spaces gives great flexibility in choosing an appropriate setting for analyzing a given equation.
\end{itemize}
In particular, the above definition of $X_{\gamma,\mu_c}$ encompasses the aforementioned properties of critical spaces from the literature.

Of special interest from a viewpoint of applications are semilinear evolution
equations of the form
\begin{equation}
\label{eq:semilin}
 \dot{u}(t) +Au(t) = G(u(t),u(t)),\; t>0,\quad u(0)=u_0.
\end{equation}
Here $G:X_\beta\times X_\beta \to X_0$ is bilinear and bounded, with the complex interpolation spaces
$ X_\beta = (X_0,X_1)_\beta$ and $A\in \calH^\infty(X_0)$ with domain
${ D}(A)=X_1$ and $\calH^\infty$-angle $\phi_{A}^\infty<\pi/2$. In this setting, $F_r=0$ and $F_s(u)=G(u,u)$, so that {\bf (H1)} is satisfied and {\bf (H2)} holds with $m=1$, $\rho_1=1$ and $\beta_1=\beta$. Hypotheses {\bf (H3)} then reads
\begin{equation}\label{eq:critweight}
2\beta-1\le \mu-1/p,
\end{equation}
so that the critical weight $\mu_c$ is given by $\mu_c=1/p+2\beta-1$ in case $\beta>1/2$ and
$$X_{\gamma,\mu_c}=(X_0,X_1)_{\mu_c-1/p,p}=(X_0,X_1)_{2\beta-1,p}$$
is the critical space.
\begin{ex}
Let $n\ge 2$, $\Omega\subset\bbR^n$ be a bounded domain with boundary $\partial\Omega$ of class $C^{3-}$, and consider the Navier-Stokes problem
\begin{equation}\label{NS}
\begin{aligned}
\partial_t u +u\cdot\nabla u-\Delta u +\nabla \pi &=0 &&\mbox{in }&& \Omega,\\
{\rm div}\, u &=0 &&\mbox{in }&&\Omega,\\
u&=0&& \mbox{on }&& \partial\Omega,\\
u(0)&=u_0&& \mbox{in }&& \Omega.
\end{aligned}
\end{equation}
Employing the Helmholtz projection $P$ in $L_q(\Omega)^n$, $1<q<\infty$, \eqref{NS} can be reformulated as the abstract semilinear evolution equation
\begin{equation}\label{ANS}
\dot{u} + Au = F(u),\; t>0,\quad u(0)=u_0,
\end{equation}
 in the Banach space
$X_0:=L_{q,\sigma}(\Omega):=P L_q(\Omega)^n$, where $A=-P\Delta$ is the Dirichlet-Stokes operator with domain
$$D(A) := \{ u\in H^2_q(\Omega)^n\cap L_{q,\sigma}(\Omega):\; u=0 \mbox{ on } \partial\Omega\}$$ and the bilinear nonlinearity $F$ is defined by
\begin{equation*}\label{F}
F(u) = G(u,u),\quad G(u_1,u_2)=-P (u_1\cdot\nabla) u_2.
\end{equation*}
In \cite{PrWi17} we computed the criticial weight $\mu_c = 1/p+ n/2q -1/2$ and the corresponding critical spaces
$$X_{\gamma,\mu_c}={_0B}^{n/q-1}_{qp}(\Omega)^n\cap L_{q,\sigma}(\Omega)$$
for \eqref{ANS}, provided $q\in (1,n)$, $p\in (1,\infty)$ such that $2/p+n/q\leq 3$. Here the subscript 0 indicates that $u|_{\partial\Omega}=0$ whenever the trace exists.
In the particular case $n=3$ and $p=q=2$ we have
$$ X_{\gamma,\mu_c} = D(A^{1/4}),\quad X_{\gamma,1}={_0H}^{1}_2(\Omega)^3\cap L_{2,\sigma}(\Omega)=D(A^{1/2}),$$
which yields the celebrated Fujita-Kato theorem, proved first in 1962 by means of the famous Fujita-Kato iteration, see \cite{FuKa62}.
%
\end{ex}
\begin{rem}
Critical spaces for the Navier-Stokes equations with perfect-slip as well as partial-slip boundary conditions haven been characterized in \cite{PrWi18}. Further examples in the context of critical spaces include
\begin{itemize}
\item the Cahn-Hilliard equation
\item the Vorticity equation
\item Convection-Diffusion equations
\item Chemotaxis equations
\end{itemize}
which can be found in \cite[Sections 3 \& 5]{PrSiWi18}. 
\end{rem}
Last but not least, let us state from \cite[Theorem 2.4]{PrSiWi18} another important application of critical spaces. It is a result of \emph{Serrin type} which connects global-in-time existence to an integral a-priori bound for the solution in the critical topology.
\begin{thm}
\label{thm:serrin}
Let $p\in (1,\infty)$, $\beta>1/2$ and $\mu_c :=2\beta -1 +1/p\leq 1$ be the critical weight.
 Assume $u_0\in X_{\gamma,\mu_c}$, and let $u$ denote the unique solution of \eqref{eq:semilin} with maximal interval of existence $[0,t^+(u_0))$. Then
\begin{enumerate}
\item[{\bf (i)}] $u\in L_p((0,a);X_{\mu_c})$, for each $a<t^+(u_0)$.
\vspace{1mm}
\item[{\bf (ii)}] If $t^+(u_0)<\infty$ then $ u\not\in L_p((0,t^+(u_0));X_{\mu_c})$.
\end{enumerate}
In particular, the solution exists globally if $u\in L_p((0,a);X_{\mu_c})$ for any finite number~$a$ with $a\le t^+(u_0)$.
\end{thm}
\noindent
We remind that $X_{\mu_c}=(X_0,X_1)_{\mu_c}$ denotes the \emph{complex} interpolation space. Theorem \ref{thm:serrin} has in particular been applied in \cite{PrWi18} and \cite{SiWi22} to prove the global well-posedness of the Navier-Stokes equations in bounded domains of $\bbR^2$ and on two-dimensional compact manifolds, respectively.

\section{Further implications and applications of weighted $L_p$-spaces}

In this section, we give a \emph{selection} of some further results and applications that were influenced by the article \cite{PrSi04}.

\subsection{General weighted spaces}

More general time-weighted function spaces with power weights have subsequently been considered for instant in \cite{AgrLinVer23}, \cite{Ama19}, \cite{LinMeyVer18}, \cite{MeySchn11}, \cite{MeyVer12} and \cite{MeyVer14}. The authors in \cite{MeyVer12} consider sharp embeddings for vector-valued weighted spaces of Besov, Triebel-Lizorkin, Bessel potential and Sobolev-Slobodeckii type.

In \cite{AgrLinVer23},  \cite{MeySchn11} and \cite{MeyVer14} anisotropic weighted function spaces have been studied. Several sharp embedding results as well as trace theorems have been proven therein. Those (anisotropic) function spaces are of significant importance in the theory of maximal regularity for parabolic boundary value problems as they appear e.g. as certain trace spaces.

Results on interpolation of weighted vector-valued function spaces with boundary conditions can be found in \cite{Ama19} and \cite{LinMeyVer18}. Those interpolation spaces appear for instance in the computation of the trace spaces or critical spaces as soon as boundary conditions come into play.

For additional literature on function spaces we recommend \cite{Ama19} \& \cite{PruSim16} and the references listed therein.

\subsection{Maximal regularity results}

Deterministic results on optimal weighted $L_p$-$L_q$-regularity for parabolic boundary value problems with inhomogeneous boundary data can be found in \cite{GalMey14}, \cite{Mey12}, \cite{MeySchn12} for the case $p=q$ and in \cite{HumLin22}, \cite{Lin20} for the general case $p\neq q$. In \cite{HumLin22}, weights in the spatial variable are considered as well. We also refer to the articles \cite{GalVer17} \& \cite{GalVer17a} concerning maximal regularity results in time-weighted spaces for parabolic operators in $\bbR^n$ having merely measurable coefficients and to the article \cite{ElsMeyReh14} for parabolic boundary value problems in non-smooth domains.

In the probabilistic setting there has been recent progress on stochastic maximal $L_p$-regularity in time-weighted spaces. In this context we want to mention the articles \cite{AgrVer22}, \cite{AgrVer22a} and \cite{PorVer19}. The first two articles can be seen as a stochastic version of \cite{PrSiWi18} concerning critical spaces, while in \cite{PorVer19}, stochastic partial differential equations with VMO coefficients within the stochastic weighted $L_p$-maximal regularity framework are considered. Finally, we want to mention the article \cite{AusNerPor14} for stochastic maximal $L_p$-regularity  results in (weighted) tent spaces. For a comprehensive list of further references in the probabilistic setting, we refer to \cite{AgrVer22} and \cite{PorVer19}.

Note that by Theorem \ref{thm:charactMR}, classical (unweighted) $L_p$-maximal regularity extrapolates to the $L_{p,\mu}$-spaces with power weights having a positive exponent. This has subsequently been extended to all Muckenhoupt weights in \cite{HaaKun07} and in \cite{ChiFio14}, \cite{ChiKro17}, \cite{ChiKro18} to more general weights. For a proof of those extrapolation results in terms of $\mathcal{R}$-boundedness, we refer to \cite{FaHyLi20}.

\subsection{Specific applications of time-weighted spaces}

The aforementioned abstract or general results on weighted $L_p$-maximal regularity have been applied to numerous concrete examples over the last years. Among the vast literature we want to mention the following references. 

Within the dynamics of fluids, weighted $L_p$-spaces and maximal regularity results have for instance been applied to 
\begin{itemize}
\item  the Hibler sea ice model \cite{BrDiHDHi22}, 
\item  a system of PDEs for magnetoviscoelastic fluids \cite{DuShaSim23}, 
\item  incompressible and inhomogeneous fluids (with variable density) \cite{FaQiZh20}, 
\item nematic liquid crystal flows via quasilinear evolution equations \cite{HiNePrSch16},
\item   the phase-field Navier-Stokes equations \cite{Kaj18}, 
\item   the Navier-Stokes equations in unbounded domains with rough initial data \cite{Kun10}, 
\item   rotating rigid bodies with a liquid-filled gap \cite{Maz21}.
\end{itemize}
For the primitive equations of geophysical flows, which might be seen as a suitable approximation of the Navier-Stokes equations, we refer to \cite{BinHie22} \& \cite{GiGrHiHuKa20} where critical spaces have been computed and analyticity of the solutions is proven. 
We further mention the articles \cite{Hie20} \& \cite{HieHus20}
for surveys ranging from boundary layers and fluid structure interaction problems over free boundary value problems and liquid crystal flow to the primitive equations.

For a comprehensive overview of applications of weighted function spaces in the context of free boundary problems as e.g. for the two-phase Navier-Stokes equations or Stefan problems, see the monograph \cite{PruSim16}. Further applications of weighted function spaces include
\begin{itemize}
\item regularity issues for the Cahn-Hilliard equation \cite{FrGaGr21},
\item reaction-diffusion systems of Maxwell-Stefan type \cite{HeMePrWi17},
\item Keller-Segel systems in critical spaces \cite{HiKrSt21},
\item bidomain operators \cite{HiePr20}.
\end{itemize}

\end{document}